\documentclass[reqno,11pt]{amsart}

\usepackage[utf8]{inputenc}
\usepackage{amsfonts}
\usepackage{amsmath}
\usepackage{cite}
\usepackage{amsthm}
\usepackage{amssymb}
\usepackage{stmaryrd}
\usepackage{mathrsfs}
\usepackage{mathtools}
\usepackage{enumitem}
\usepackage{comment}
\usepackage{pdflscape}
\usepackage{hyperref}

\usepackage{rotating}

\usepackage{empheq}

\usepackage{color}
\usepackage{tikz}
\usetikzlibrary{shapes,decorations.pathreplacing}
\usetikzlibrary{calc}
\usetikzlibrary{backgrounds}
\usetikzlibrary{automata}
\usetikzlibrary{positioning}
\usetikzlibrary{decorations.markings}
\usetikzlibrary{arrows}
\usetikzlibrary{scopes}
\usetikzlibrary{intersections}
\usetikzlibrary{topaths}
\usepackage{empheq}
\numberwithin{equation}{section}

\newtheorem{theorem}{Theorem}[section]
\newtheorem{lemma}[theorem]{Lemma}
\newtheorem{proposition}[theorem]{Proposition}
\newtheorem{corollary}[theorem]{Corollary}

{\theoremstyle{definition}
}
{\theoremstyle{remark}
\newtheorem{remark}[theorem]{Remark}}
\theoremstyle{remark}{
\newtheorem{example}[theorem]{Example}}

\newcommand{\Z}{\mathbb{Z}}

\newcommand{\Q}{\mathbb{Q}}

\newcommand{\tree}{\mathcal{T}}
\newcommand{\atm}{\mathcal{A}}

\DeclareMathOperator{\Stab}{Stab}
\DeclareMathOperator{\Aut}{Aut}

\definecolor{darkgreen}{rgb}{0,0.6,0}

\usepackage{ifthen}
\newcommand{\optionalarg}[2]{
\ifthenelse{\equal{#2}{}}{%
#1}{%
#1(#2)}
}

\begin{document}

\title[Lamplighter groups, bireversible automata and rational series]{Lamplighter groups, bireversible automata and rational series over finite rings}
\date{\today}


\author[R.~Skipper]{Rachel Skipper}
\address{Department of Mathematics, The Ohio State University, Columbus, OH 43210}
\email{skipper.26@osu.edu}

\author[B.~Steinberg]{Benjamin Steinberg}
\address{Department of Mathematics, City College of New York, New York, NY  10031}
\email{bsteinberg@ccny.cuny.edu}

\thanks{The first author was partially supported by a grant from the
Simons Foundation (\#245855 to Marcin Mazur). The second author was supported by NSA MSP \#H98230-16-1-0047 and a PSC-CUNY grant}

\begin{abstract}
We realize lamplighter groups $A\wr \mathbb Z$, with $A$ a finite abelian group, as automaton groups via affine transformations of power series rings with coefficients in a finite commutative ring. Our methods can realize $A\wr \mathbb Z$ as a bireversible automaton group if and only if the $2$-Sylow subgroup of $A$ has no multiplicity one summands in its expression as a direct sum of cyclic groups of order a power of $2$.
\end{abstract}

\maketitle
%\thispagestyle{empty}

%--------------------------------------------------------------------
\section{Introduction}\label{sec:intro}
In the process of classifying all two-state automaton groups over a two-letter alphabet, Grigorchuk and \.{Z}uk discovered that the lamplighter group $\mathbb Z/2\Z\wr \Z$ can be realized as an automaton group~\cite{gz01}.  Moreover, they used this realization to compute the spectral measure for a simple random walk on the group with respect to the generating set arising from the automaton.  This in turn led to the first counterexample to the strong form of the Atiyah conjecture on $\ell_2$-Betti numbers~\cite{GLSZ00}.  The proof of Grigorchuk and \.{Z}uk that their automaton generated a lamplighter groups was computational, making use of the wreath product representation of automaton groups.  A more conceptual proof, using affine transformations of the power series ring over the two-element field, appeared in~\cite{GNS00}.  This paper also realized lamplighter groups of the form $(\Z/p\Z)^n\wr \Z$ with $p$ prime using series over fields. Here, by a lamplighter group we mean a group of the form $F\wr \Z$ where $F$ is a finite group.

Dicks and Schick computed spectral measures for random walks on arbitrary lamplighter groups $F\wr \Z$ with respect to a set of generators inspired by the automaton generators of Grigorchuk and \.{Z}uk~\cite{DS02}.  With respect to this generating set, the Cayley graph of $F\wr \Z$ depends (up to isomorphism) only on $|F|$ and not the group structure of $F$.  Their approach avoided automata completely.  The second author and Silva realized all lamplighter groups $A\wr \Z$ with $A$ a finite abelian group in~\cite{ss05} by using affine transformations of power series rings over finite commutative rings.  The automaton generators in this case are exactly those considered by Dicks and Schick and so the second author, together with Kambites and Silva,   exploited this automaton realization in~\cite{KSS06} to give a new proof of the results of Dicks and Schick using the original scheme of Grigorchuk and \.{Z}uk.
We mention another early paper realizing lamplighter groups was~\cite{BS06}.

Since then there has been further study of automaton group realizations of lamplighter groups, particularly in connection with power series.  For example, Savchuk and Sidki~\cite{SS16} have recently studied realizations of lamplighter groups as affine transformations over power series rings with coefficients in $\Z/d\Z$ and they provide a detailed study of all such representations over the binary tree.  The lamplighter group $\mathbb Z/2\Z\wr \Z$ is represented as a bounded automaton group in~\cite{JSW17}.  The recent paper~\cite{BS18} constructs $A\wr \Z^d$ as an automaton group for any finite abelian group $A$ and $d\geq 1$.

An important class of automaton groups is the class of bireversible automaton groups~\cite{MNS00, nek05}.  These are the automaton groups for which both the automaton semigroups of the dual automaton and the dual of the inverse automaton are actually groups.  One salient feature of bireversible automaton groups is that they automatically act essentially freely on the boundary of the rooted tree~\cite{svv11} and hence, by the results of~\cite{KSS06}, one can potentially compute spectral measures for their random walks via the action on the tree.  Bireversible automata were used for the first constructions of free groups and virtually free groups as automaton groups~\cite{GM05,nek05,vv07,vv10,sv11,svv11}. Note that if a group generated by a bireversible automaton contains an element of infinite order, then it contains a nonabelian free subsemigroup and thus has exponential word growth~\cite{Kli18, FM18}.

%For a long time, most of the examples of bireversible automaton groups were virtually free, with the exception of some lattices in the automorphism group of a product of two trees that were realized as bireversible automata groups by Glasner and Mozes~\cite{GM05}.

For a long time, most of the examples of bireversible automaton groups were virtually torsion-free and so it was quite surprising when Bondarenko, D’Angeli and Rodaro~\cite{bdr16} found a bireversible automaton group isomorphic to $\Z/3\Z\wr \Z$.  Their original proof was using wreath product representations but Bondarenko and Savchuk have recently announced a construction via affine transformations of the ring of power series over $\Z/3\Z$. Multiplication by a rational power series over a field can always be implemented by a finite state automaton.  Bondarenko and Savchuk announced a complete description of all automaton groups generated by the states of such automata (they are always lamplighter groups if the series is invertible and non-constant), as well as their dual automaton semigroups.  In particular, they characterize when such automaton groups are bireversible. Such an approach was also used by Ahmed and Savchuk~\cite{AS18} to realize $(\Z/2\Z)^2\wr \Z$ as a bireversible automaton group.

In this paper, like in~\cite{ss05}, we consider rational power series over a finite commutative ring $R$.  More specifically, we consider invertible rational series of the form
\begin{equation}\label{eq:ourseries}
f(t)=r\left(\frac{1-at}{1-bt}\right)=r(1-at)\cdot \sum_{n=0}^{\infty}b^nt^n
\end{equation}
 where $r\in R^\times$ is a unit and $a,b\in R$. These are precisely the invertible rational series, which together with their inverses, have corresponding recurrence of degree at most $1$.  We explicitly construct an initial automaton $\atm_f$ computing left multiplication by $f(t)$.  Here we identify $R^{\omega}$ with $R\llbracket t\rrbracket$ in the obvious way.  We then consider the automaton group generated by the states of $\atm_f$.  We prove that if $a-b$ is a unit of $R$, then we obtain $R^+\wr \Z$, where $R^+$ is the additive group of $R$.  The special case where $r=1$, $a=0$ and $b=1$ recovers the construction in~\cite{ss05}.  Most likely if $a-b$ is not a unit, one never gets the lamplighter group $R^+\wr \Z$, but we prove this negative result explicitly only when $R=\Z/n\Z$.

The main result is that $\atm_f$ is a bireversible automaton generating a group isomorphic to $R^+\wr \Z$ if and only if $a$, $b$ and $a-b$ are units of $R$.  This then leads to the natural question of which finite abelian groups $A$ can be obtained as the additive group of a finite commutative ring $R$ with two units whose difference is a unit.  Note that $R$ has this property if and only if it has no ideal of index $2$.  Not all finite abelian groups can be realized this way.  For example $\Z/n\Z$ with $n$ even can never be realized as the additive group of such a ring.  We prove that $A$ is isomorphic to the additive group of a ring $R$ with no ideal of index $2$ if and only if the $2$-Sylow subgroup of $A$ can be expressed as a direct sum of cyclic groups of the form $(\Z/2^i\Z)^{r_i}$ with $r_i\geq 2$. Consequently, we can realize lamplighters $A\wr \Z$ as bireversible automaton groups whenever the $2$-Sylow subgroup of $A$ satisfies this condition.  We are hopeful that the extra symmetry appearing in a bireversible automaton might make it possible to perform spectral computations for these lamplighter groups with respect to other probability measures than the uniform measure on the automaton generating set.

The paper is organized as follows.  We begin by recalling some basic notions concerning automata, automaton groups and actions on rooted trees.  The next section discusses the action of multiplication by $f(t)$ from \eqref{eq:ourseries} on power series from this viewpoint and constructs its minimal automaton.  We study the group generated by the states of this automaton and prove that it is a lamplighter group when $a-b$ is a unit of $R$.  Before doing this, we review some of the basic theory of finite commutative rings, e.g., that they are direct products of local rings.  We also characterize reversibility and bireversibility of the automaton.  The following section characterizes the additive groups of finite commutative rings containing two units whose difference is a unit.  The final section provides some examples of our construction.

\section{Automata}\label{sec:background}
For background on automata groups, the reader is referred to the book of Nekrashevych~\cite{nek05} or the survey paper~\cite{GNS00}.
Let $X$ be a finite set called an \emph{alphabet} and let $X^*$ be the free monoid on $X$, that is, the set of finite words over the alphabet $X$ including the empty word, denoted $\varnothing$. The length of  a word $w\in X^*$ is denoted by $|w|$.  The Cayley graph of $X^*$ naturally has the structure of a regular rooted tree with $\varnothing$ as the root and two words $u$ and $v$ are connected by an edge if $ux=v$ or $u=vx$ for some $x\in X$, and so we denote the Cayley graph by $\tree_X$.

An \emph{automorphism of $\tree_X$} is a bijection from $X^*$ to $X^*$ which preserves edge incidences. The group of all automorphisms of $\tree_X$ is denoted $\Aut(\tree_X)$. Any automorphism of $\tree_X$ induces an action on the boundary of the tree which can be identified with the set of infinite words over $X$ and is denoted $X^\omega$.  Conversely, any permutation of $X^\omega$ which preserves the length of the longest common prefix of any pair of infinite words is induced by the action of a unique automorphism of $\tree_X$.

For $v\in X^*$ and $u\in X^\omega$, an automorphism $g\in \Aut(\tree_X)$ acts on $vu$ via
\[g(vu)=g(v)g|_v(u)\]
where $g(v)\in X^*$ with $|g(v)|=|v|$ and  $g|_v\in \Aut(\tree_X)$, depending only on $v$. We call $g|_v$ the \emph{state} of $g$ at $v$ (some authors use the term ``section''). An automorphism $g$ of $\tree_X$ is said to be \emph{finite state} if the set $\{g|_v : v\in X^*\}$ is finite. The following is a straightforward computation, taking the action by automorphisms to be a left action.

\begin{lemma}\label{lem:states}
Let $g,h\in \Aut(\tree_X)$.  Then the states of $gh$ and $g^{-1}$ at $v\in X^*$ are given by
\begin{align*}
(gh)|_v &=g|_{h(v)}h|_v\\
(g^{-1})|_v &= (g|_{g^{-1}(v)})^{-1}.
\end{align*}
\end{lemma}

It follows from the lemma that a composition of finite state automorphisms is again finite state and the inverse of a finite state automorphism is finite state.

A subgroup $G\leq \Aut(\tree_X)$ is said to be \emph{self-similar} if, for all $g\in G$ and $v\in X^*$, $g|_v$ is in $G$. For a vertex $v$, the stabilizer of $G$ at $v$, denoted $\Stab_G(v)$, is the set of $g\in G$ with $g(v)=v$. The group $G$ is \emph{self-replicating} if \[\{g|_v : g\in \Stab_G(v)\}=G\] for all $v\in X^*$ (actually, it is enough for all $v\in X$).
 It is \emph{spherically transitive} if for every $v$ and $w$ of same length over the alphabet, there exists $g\in G$ with $g(v)=w$.  Note that some authors include spherical transitivity as part of the definition of self-replicating.

A self-similar group is said to be an $\emph{automaton group}$ if it is finitely generated and every element of the group is finite state. In this case, a finite state Mealy automaton can be used to describe the generators.

A \emph{(Mealy) automaton} is $\atm$ is a 4-tuple $\atm=(Q, X, \delta, \lambda)$ where $Q$ is a finite set of states, $X$ is a finite alphabet, $\delta\colon  Q\times X \rightarrow Q$ is the \emph{transition function}, and $\lambda\colon  Q \times X \rightarrow X$ is the \emph{output function}. For each $q\in Q$ and $x\in X$, we will use the notation $\lambda_q(x)$ to mean $\lambda(q,x)$ and will call $\lambda_q$ the \emph{state function corresponding to $q$}. Similarly, for $x\in X$, we define $\delta_x\colon Q\to Q$ by  $\delta_x(q)=\delta(q,x)$.  When the set of states is finite, we say $\atm$ is a \emph{finite state automaton}.

It is common to describe an automaton $\atm$ by a directed, labeled graph with vertices labeled by $Q$ and edges
\[q \xrightarrow{\,\,x\, \mid\, \lambda_q(x)\,} \delta(q,x)\]
for each $q\in Q$ and $x\in X$.

The function $\lambda_q$ can be extended uniquely to a function on both the sets $X^*$ of finite words and $X^\omega$ of infinite words over $X$, which, abusing notation, we shall also refer to as $\lambda_q$. These extensions are described recursively as follows:
\[\lambda_q(x_0x_1 \cdots x_n)=\lambda_q(x_0)\lambda_{\delta(q,x_0)}(x_1 \cdots x_n)\]
and
\[\lambda_q(x_0x_1 \cdots)=\lim_{n\rightarrow \infty}\lambda_q(x_0x_1 \cdots x_n).\]  Intuitively, the input word labels the left hand side of a unique path starting at $q$ in the directed, labeled graph representing $\atm$ and the output is the label of the right hand side of this path.

For each $g\in \Aut(\tree_X)$, there is a unique minimal automaton $\atm_g$ and state $q$ such that $g=\lambda_q$.  The state set is given by $\{g|_v: v\in X^*\}$, the transition function is given by $\delta(g|_v,x)=g|_{vx}$ and the output function is given by $\lambda(g|_v,x) = g|_v(x)$.  One then has $g=\lambda_{g|_{\varnothing}}$.  Moreover, $\atm_g$ is finite if and only if $g$ is finite state.

An automaton $\atm$ is \emph{invertible} if, for each $q$, $\lambda_q$ is a permutation of the alphabet. Invertible automata are precisely the automata for which each state function describes an automorphism of the tree $\tree_X$ with vertex set $X^*$. For an automaton $\atm=(Q, X, \delta, \lambda)$, the \emph{inverse automaton} $\atm^{-1}$ is obtained by switching the input and output letters on the edge labels. In this case, the inverse to $\lambda_q$ is computed by the state corresponding to $q$ in $\atm^{-1}$.  Note that, for $g\in \Aut(\tree_X)$, the minimal automaton for $g^{-1}$ is the inverse of the minimal automaton for $g$, as is easily seen from Lemma~\ref{lem:states}.

The \emph{dual automaton} $\partial \atm$ is given by $(X, Q, \lambda, \delta)$, i.e., the alphabet and states are interchanged and the output and transition functions are interchanged. An invertible automaton is called \emph{reversible} if its dual is invertible and \emph{bireversible} if it is reversible and additionally its inverse is reversible. Note that some authors do not require $\atm$ to be invertible in order to be reversible. The state functions of $\partial \atm$ are precisely the functions $\delta_x$, with $x\in X$, and hence $\atm$ is reversible if and only if $\delta_x$ is a permutation of the state set for each $x\in X$.

For an invertible finite state automaton $\atm$, the group generated by $\{\lambda_q : q\in Q \}$ under the operation of composition is called the \emph{automaton group} $\mathbb{G}(\atm)$.  It is a self-similar group and it is an automaton group in the sense defined above.  The group generated by a bireversible automaton enjoys the property that its action on $X^{\omega}$ is always essentially free~\cite{svv11}, whereas the action of a reversible automaton is sometimes essentially free and sometimes not~\cite{KSS06}.  The action  on $X^\omega$ is \emph{essentially free} if the stabilizer of an infinite word is almost surely trivial with respect to the Bernoulli measure.  This is exactly the property that one needs for the spectral measure of the simple random walk on the Cayley graph to be computable as the limit of the Kesten-von Neumann-Serre spectral measures~\cite{GZ04} for the random walks on the levels of the tree in the spherically transitive case, see~\cite{KSS06} for further details.

\section{Power series and lamplighter groups}
By a \emph{lamplighter group} we mean a restricted wreath product $F\wr \mathbb Z=\bigoplus_{\mathbb Z}F\rtimes \mathbb Z$ with $F$ a finite group.  Some authors refer to only the particular case $(\mathbb Z/2\mathbb Z)\wr \mathbb Z$ as the lamplighter group.
In this section we use the language of formal power series to produce automata which generate lamplighter groups of the form $A\wr \mathbb Z$ with $A$ a finite abelian group.  This was first done for $A=\mathbb Z/2\mathbb Z$ by Grigorchuk and \.{Z}uk~\cite{gz01} (where power series were not used, but see~\cite{GNS00} for a proof using power series) and for arbitrary finite abelian groups by the second author and Silva~\cite{ss05}. We find conditions on the power series that will guarantee that the automaton is reversible or bireversible.  Related work, via power series over fields, has been announced by Bondarenko and Savchuk at various conferences.  Their work was announced before ours, and considers rational power series with higher order recurrences than ours.  However, by working over rings we can realize many more lamplighter groups as bireversible automata than can be realized over a field.

Let $R$ be a finite commutative ring with unity, $R^+$ its additive group, $R^\times$ its multiplicative group of units, and $R\llbracket t\rrbracket$ the ring of formal power series with coefficients in $R$.  Note that $f\in R\llbracket t\rrbracket$ is a unit if and only if its constant term $f(0)$ is a unit of $R$. We will identify $R^\omega$ with $R\llbracket t\rrbracket$ via
\[(r_0,r_1,r_2,\ldots) \longmapsto r_0 +r_1t+r_2t^2+\cdots\] where on the left hand side we write elements of $R^{\omega}$ as tuples to avoid confusion between concatenation of words and the multiplication in $R$.

For any power series $f(t)$ in $R\llbracket t\rrbracket$ we define two mappings of $R^\omega$ given by
\[ \mu_f\colon  g(t)\longmapsto f(t)g(t)\]
and
\[\alpha_f\colon  g(t)\longmapsto f(t)+g(t).\]

Note that $\alpha_f$ is invertible with inverse $\alpha_{-f}$ and preserves the length of the longest common prefix. Thus $\alpha_f$ gives an automorphism of the tree $\tree_R$ with vertex set $R^*$. On the other hand, $\mu_f$ is invertible precisely when $f$ is a unit of $R\llbracket t\rrbracket$. In this case $(\mu_f)^{-1}=\mu_{f^{-1}}$ and both these mappings preserve the length of the longest common prefix.  It is well known to automata theorists that if $f$ is a rational power series, that is, $f(t)=p(t)/q(t)$ with $p(t),q(t)$ polynomials and $q(0)\neq 0$, then $\alpha_f,\mu_f$ are finite state. We shall compute the minimal automaton for $\mu_f$ in the case that $p(t),q(t)$ are linear, momentarily.

 The following proposition illustrates the relationship between $\mu$ and $\alpha$.

\begin{proposition}\label{prop:conj}
For any $n\in \Z$ and $f(t), h(t)\in R\llbracket t\rrbracket$, $\mu_{f}\alpha_h\mu_{f^{-1}}=\alpha_{fh}$.
\end{proposition}
\begin{proof}
If $g(t)\in R\llbracket t\rrbracket$, then we compute $\mu_{f}\alpha_h\mu_{f^{-1}}(g)=\mu_f\alpha_h(f^{-1}g)=\mu_f(f^{-1}g+h)=g+fh=\alpha_{fh}(g)$.
\end{proof}

For any power series $f(t)=\sum_{i=0}^\infty c_i t^i$ we define also the \emph{shift of $f$} by \[\sigma(f)=\sum_{i=0}^\infty c_{i+1} t^i=c_1+c_2t+c_3t^2+\cdots\] so that $f(t)=c_0+\sigma(f)t$.

We wish to study rational power series of the form \[f(t)=\frac{r_1-r_2t}{r_3-r_4t}.\] In order for $f(t)$ to be a unit of $R\llbracket t\rrbracket$, $r_1$ and $r_3$ must be units of $R$, and so we will write $f(t)$ as \[f(t)= r\left(\frac{1-at}{1-bt}\right)\] for $a$ and $b$ in $R$ and $r\in R^\times$. With this notation
\[f(t)=r +r(b-a)t+ rb(b-a)t^2+rb^2(b-a)t^3+rb^3(b-a)t^4+\cdots\] as is easily checked.

\begin{lemma}\label{lem:shift}
Let \[f(t)=r\left(\frac{1-at}{1-bt}\right)\] where $r\in R^\times$ and $a,b\in R$. Then $\sigma(f)=bf-ra$.
\end{lemma}
\begin{proof}
Observe that
\begin{align*}
bf-ra &=b(r +r(b-a)t+ rb(b-a)t^2+rb^2(b-a)t^3+\cdots)-ra\\
&=br-ra + rb(b-a)t+rb^2(b-a)t^2+ rb^3(b-a)t^3 +\cdots\\
&=r(b-a)+ rb(b-a)t+rb^2(b-a)t^2+ rb^3(b-a)t^3+\cdots\\
&=\sigma(f).
\end{align*}
\end{proof}

\begin{proposition}\label{prop:statesoff}
Let \[f(t)=r\left(\frac{1-at}{1-bt}\right)\] where $r\in R^\times$ and $a,b\in R$. Then $\mu_f$ is finite state with set of states $\{\alpha_{-sra}\mu_f\alpha_{sb}: s\in R\}$. Moreover, for any $s\in R$, the state $\alpha_{-sra}\mu_f\alpha_{sb}$ permutes the degree zero terms via:
\[\tilde{s} \longmapsto r(\tilde s+(b-a)s).\]
\end{proposition}
\begin{proof}
Let $g(t)=\sum_{i=0}^\infty d_it^i=d_0 +\sigma(g)t$. Then by Lemma~\ref{lem:shift}
\begin{align*}
\mu_f(g(t))&=(r+\sigma(f)t)(d_0+\sigma(g)t)\\
&=rd_0+ f\sigma(g)t+d_0\sigma(f)t\\
&=rd_0+f\sigma(g)t+d_0(bf-ra)t\\
%&=rd_0 +f\sigma(g)t+d_0bft-d_0rat\\
&=rd_0 +[-d_0ra+f\sigma(g)+fd_0b]t\\
&=rd_0+[-d_0ra+f(\sigma(g)+d_0b)]t\\
&=rd_0+[\alpha_{-d_0ra}\mu_f\alpha_{d_0b}(\sigma(g))]t.
\end{align*}
Therefore, the state of $\mu_f$ at the vertex $d_0$ is given by $\alpha_{-d_0ra}\mu_f\alpha_{d_0b}$ and is of the desired form and $\mu_f(d_0)=rd_0$.

Now it just remains to show that the states of $\alpha_{-sra}\mu_f\alpha_{sb}$ for words of length $1$ have the same form.
Note that the states of $\alpha_{-sra}$ and $\alpha_{sb}$ at every word of length greater than $0$ are trivial since, for any $d\in R$, we have
\[\alpha_d(c_0+c_1t+c_2t^2+\cdots) = (d+c_0)+c_1t+c_2t^2+\cdots .\] Applying Lemma \ref{lem:states}, we see that, for any $d_0\in R$,
\begin{align*}
(\alpha_{-sra}\mu_f\alpha_{sb})|_{d_0}&=(\alpha_{-sra}\mu_f)|_{d_0+sb}\\
&=(\alpha_{-sra})|_{r(d_0+sb)}(\mu_f)|_{d_0+sb}\\
&=\alpha_{-(d_0+sb)ra}\mu_f\alpha_{(d_0+sb)b}.
\end{align*}

Therefore, the state of $\alpha_{-sra}\mu_f\alpha_{sb}$ at $d_0$ given by $\alpha_{-(sb+d_0)ra}\mu_f\alpha_{(sb+d_0)b}$ is again  of the desired form.

Finally, it is straightforward to check that if $g(t)$ is a power series with constant term $\tilde{s}$ then the constant term of $\alpha_{-sra}\mu_f\alpha_{sb}(g)$ is $r(\tilde s+(b-a)s)$.  Indeed, since $\mu_f$ acts on the root of $\tree_R$ via multiplication by $r$, we have that $\alpha_{-sra}\mu_f\alpha_{sb}(\tilde s) = -sra+r(\tilde s+sb)=r(\tilde s+(b-a)s)$.
\end{proof}

The above proposition tells us that we can associate to $\mu_f$ a finite state automaton whose state functions are given by the states of $\mu_f$. In other words, define $\atm_f=(Q, X, \delta, \lambda)$ with states $Q=\{\alpha_{-sra}\mu_f\alpha_{sb}: s\in R\}$ and alphabet $X=R$. The transition function $\delta$ is given by
\begin{equation}\label{eq:transition.func}
\delta(\alpha_{-sra}\mu_f\alpha_{sb}, \tilde{s})=\alpha_{-(sb+\tilde{s})ra}\mu_f\alpha_{(sb+\tilde{s})b}
\end{equation}
 and the output function by
\begin{equation}\label{eq:output.func}
\lambda(\alpha_{-sra}\mu_f\alpha_{sb}, \tilde{s})=r(\tilde s+(b-a)s).
\end{equation}
The state function $\lambda_{\mu_f}$ is precisely $\mu_f$ and, more generally, $\lambda_{\alpha_{-sra}\mu_f\alpha_{sb}}=\alpha_{-sra}\mu_f\alpha_{sb}$ for $s\in R$.

We now review some basic properties of finite commutative rings. It is well known that if a commutative ring with unity is Artinian, then it is a finite direct product of Artinian commutative local rings (cf.~\cite[Cor.~2.16]{Eis95}). Since $R$ is a finite ring it is clearly Artinian and so we shall write $R= R_1 \times \cdots \times R_n$ where $R_1, \ldots, R_n$ are local rings. Let $e_i=(0, \ldots, 0, 1, 0, \ldots, 0)$ with $1$ in the $i$-th position so that $\sum_{i=1}^n e_i=1$. In fact, $e_1,\ldots, e_n$ form a complete set of orthogonal primitive idempotents of $R$.

For any $R$-module $M$, let $M_i=e_iM$; it is an $R_i$-module. Then $M= M_1\times \cdots \times M_n$ and the action of $R$ is coordinatewise: \[(r_1,\ldots,r_n)(m_1,\ldots, m_n) = (r_1m_1,\ldots,r_nm_n).\] It is clear that a subset $A$ of $M$ is linearly independent over $R$ if and only if, for each $i$, its projection $e_iA$ into $M_i$ is linearly independent over $R_i$.

Recall now that, for a commutative Artinian local ring $S$, the unique maximal ideal $\mathfrak{m}$ consists of the nilpotent elements of $S$ and that $S^\times=S\backslash \mathfrak{m}$.

\begin{proposition}\label{prop:fmlinind}
Let \[f(t)=r\left(\frac{1-at}{1-bt}\right)\] where $r\in R^\times$ and $a,b\in R$. Then the set $\{f^m : m\in \Z\}$ is linearly independent over $R$ if and only if $a-b\in R^\times$.
\end{proposition}
\begin{proof}
Write $R=R_1\times \cdots \times R_n$ where each $R_i$ is local. For $s\in R$, let $s_i$ be the image of $s$ in $R_i$ for $1\leq i \leq n$. Note that $a-b\in R^\times$ if and only if $a_i-b_i\in R_i^\times$ for all $i$ and so we may assume without loss of generality that $R$ is local with maximal ideal $\mathfrak{m}$ by the observation preceding the proposition.

Suppose first that $a-b\notin R^\times$ and so $a-b\in \mathfrak{m}$. Since $\mathfrak{m}$ consists of nilpotent elements, we can find $k\geq 1$ with $(a-b)^k=0$ and $(a-b)^{k-1}\neq 0$. Since \[f(t)=r\left(\frac{1-at}{1-bt}\right)=r\left(1-\frac{(a-b)t}{1-bt}\right)\] we get that $(a-b)^{k-1}f=r(a-b)^{k-1}$. Therefore, $(a-b)^{k-1}f-r(a-b)^{k-1}f^0=0$ and $(a-b)^{k-1},r(a-b)^{k-1}\neq 0$.  Thus the powers of $f$ are linearly dependent over $R$.

Assume now that $a-b\in R^\times$. Then either $a$ or $b$ is not in $\mathfrak{m}$. Without loss of generality, assume $a\in R^\times$. Since $f$ is a unit, if $cf^n=0$ with $c\in R$, then $c=0$.  Suppose that $c_1f^{n_1}+\cdots+c_kf^{n_k}=0$ with $n_1< n_2 < \cdots < n_k$, $c_i\neq 0$ for all $i$, and $k>1$. Multiplying the expression by $f^{-n_1}$ if necessary, we may assume $n_1=0$. Now multiply by $(1-bt)^{n_k}$ to get
\[c_1(1-bt)^{n_k}+c_2r^{n_2}(1-bt)^{n_k-n_2}(1-at)^{n_2}+\cdots +c_kr^{n_k}(1-at)^{n_k}=0.\]
Putting $t=\frac{1}{a}$, we get that $c_1(1-\frac{b}{a})^{n_k}=0$ and hence $c_1(\frac{a-b}{a})^{n_k}=0$. Since $a-b$ and $a$ are both units we get that $c_1=0$, a contradiction. Therefore, $\{f^m : m\in \Z\}$ is linearly independent over $R$.
\end{proof}

\begin{proposition}
\label{prop:linind}
Let \[f(t)=r\left(\frac{1-at}{1-bt}\right)\] where $r\in R^\times$ and $a,b\in R$. Then the set $B=\{(-ar+bf)f^m : m\in \Z\}$ is linearly independent over $R$ if and only if $a-b\in R^\times$.  Moreover, if $a-b\notin R^\times$, then there exists $s\in R\setminus \{0\}$ with $sB=0$.
\end{proposition}
\begin{proof}
First note that
\[-ar+bf = \frac{-ar(1-bt)+ br(1-at)}{1-bt} = \frac{-r(a-b)}{1-bt}.\]  Thus
\begin{equation}\label{equation:realbasis}
B=\left\{-r(a-b)\left(\frac{1}{1-bt}\right)f^m:m\in \Z\right\}.
\end{equation}
  If $a-b$ is not a unit of $R$, then since $R$ is a finite, there exists $s\in R\setminus \{0\}$ with $s(a-b)=0$.  It follows that $sB=0$ and hence $B$ is not linearly independent over $R$.  Conversely, if $a-b\in R^\times$, then $-r(a-b)\left(\frac{1}{1-bt}\right)$ is a unit of $R\llbracket t\rrbracket$.  The linear independence of $B$ over $R$ is then immediate from the linear independence of $\{f^m: m\in \Z\}$ from Proposition~\ref{prop:fmlinind}.
%Again, without loss of generality, we assume $R$ is local with maximal ideal $\mathfrak{m}$.
%
%If $(a-b)\in \mathfrak{m}$, then as in the last proof we can choose  $k\geq 1$ with $(a-b)^k=0$ but $(a-b)^{k-1}\neq 0$. Then
%\begin{align*}
%(-ar+bf)(a-b)^{k-1}&=-ar(a-b)^{k-1}+br\left(1-\frac{(a-b)t}{1-bt}\right)(a-b)^{k-1}\\
%&=-ar(a-b)^{k-1}+br(a-b)^{k-1}\\
%&=-r(a-b)^k=0
%\end{align*}
%but $(a-b)^{k-1}\neq 0$.  Thus the above set is not linearly independent.%
%
%Now suppose that $(a-b)\in R^\times$. Then by Proposition \ref{prop:fmlinind} $\{f^m : m\in \Z \}$ is linearly independent. Let $c_1(-ar+bf)f^{n_1}+\cdots+ %c_k(-ar+bf)f^{n_k}=0$ with $n_1<\cdots < n_k$ and $c_i\neq 0$ for all $i$. Since $(a-b)\notin \mathfrak{m}$, either $a$ or $b$ is a unit and so we have two cases.
%
%If $a$ is a unit, multiplying the expression by $f^{-n_1}$ if necessary, we assume $n_1=0$. Since $\{f^m : m\in \Z \}$ is linearly independent we get that the coefficient %$-arc_1$ of $f^0$ is equal to $0$. Since $a$ and $r$ are units, we get $c_1=0$, a contradiction.
%
%If instead $b$ is a unit, multiplying the expression by $f^{-n_k-1}$ if necessary, we assume $n_k=-1$. Since $\{f^m : m\in \Z \}$ is linearly independent, the coefficient %$bc_{k}$ of $f^0$ is equal to $0$. Since $b$ is a unit, we get $c_k=0$, a contradiction.
\end{proof}

%\begin{remark}\label{rem:torsion}
%Note that if $R=\mathbb Z/n\mathbb Z$ and $a-b$ is not a unit, then the abelian group generated by $\{(-ar+bf)f^m : m\in \Z\}$ is annihilated by some positive integer %strictly less than $n$ by Proposition~\ref{prop:linind}.  We shall see later that this implies $\mathbb G(\atm_f)$ is not isomorphic to $R^+\wr \mathbb Z$ in this case.
%\end{remark}

We remark that if $f(t)=1-2t$ in $(\mathbb Z/4\mathbb Z)\llbracket t\rrbracket$, then $f^2=1$ and so $\mathbb G(\atm_f)$ is in fact a finite group.  In this case $a=2$, $b=0$ and $a-b$ is not a unit.

\begin{theorem}\label{thm:lamplighter}
Let \[f(t)=r\left(\frac{1-at}{1-bt}\right)\] where $r\in R^\times$ and $a,b\in R$. If $a-b\in R^\times$, then $\mathbb{G}(\atm_f)=\langle\alpha_{-sra}\mu_f\alpha_{sb}: s\in R \rangle \cong R^+\wr \Z$.
\end{theorem}
\begin{proof}
Since $\alpha_{-sra}\mu_f\alpha_{sb}=\alpha_{-sra+sbf}\mu_f=\alpha_{s(-ar+bf)}\mu_f$ by Proposition~\ref{prop:conj}, we can take $\{\alpha_{s(-ar+bf)}, \mu_f: s\in R\}$ as a generating set for $\mathbb{G}(\atm_f)$.

By Proposition \ref{prop:linind}, the set $\{(-ar+bf)f^m : m\in \Z\}$ is a linearly independent set over $R$ and forms an infinite dimensional basis for a free module over $R$. Therefore, by Proposition~\ref{prop:conj},
\[N=\langle \alpha_{s(-ar+bf)f^m} : m \in \Z \rangle =\langle (\mu_f)^m\alpha_{s(-ar+bf)}(\mu_f)^{-m}:m\in \mathbb Z\rangle \cong \bigoplus_{\Z} R^+.\]
Moreover, $N$ is clearly normal in $\mathbb{G}(\atm_f)$ and contains $\alpha_{s(-ar+bf)}$ for all $s\in R$. Furthermore, $N$ intersects $\langle \mu_f \rangle$ trivially (as Proposition~\ref{prop:fmlinind} implies $\mu_f$ has infinite order) and $\mu_f$ acts on $N$ via a shift. We conclude that
\[\mathbb{G}(\atm_f)\cong \bigoplus_{\Z} R^+ \rtimes \Z= R^+\wr \Z\]
as desired.
\end{proof}

\begin{remark}
Note that if $R=\mathbb Z/n\mathbb Z$ and $a,b\in R$ with $a-b$ not a unit, then either $\mu_f$ has finite order (and hence $\mathbb G(\atm_f)$ is finite) or, by Proposition~\ref{prop:linind}, the torsion subgroup of $\mathbb{G}(\atm_f)$, that is, the subgroup $\langle \alpha_{\tilde s(-ar+bf)f^m} : m \in \Z, \tilde s\in R\rangle$ is annihilated by some $0<s<n$ hence is not isomorphic to $\bigoplus_{\mathbb Z} R^+$.   Thus $\mathbb G(\atm_f)$ is not isomorphic to $R^+\wr \mathbb Z$ in this case.
\end{remark}

\begin{lemma}\label{lem:numberofstates}
Let \[f(t)=r\left(\frac{1-at}{1-bt}\right)\] where $r\in R^\times$, $a,b\in R$ and $a-b$ is a unit. Then $\atm_f$ has $|R|$ states.
\end{lemma}

\begin{proof}
Note that, for $s\in S$, we have $\alpha_{-sra}\mu_f\alpha_{sb} = \alpha_{-sra+sbf}\mu_f$ by Proposition~\ref{prop:conj}.  Since $a-b$ is a unit and $\mu_f$ is invertible, we see from Proposition~\ref{prop:linind} that distinct elements $s\in R$ give rise to distinct states of $\atm_f$.
\end{proof}

\begin{theorem}\label{thm:revbirev}
Let \[f(t)=r\left(\frac{1-at}{1-bt}\right)\] where $r\in R^\times$, $a,b\in R$ and $a-b$ is a unit. Then
\begin{enumerate}
\item $\atm_f$ is reversible if and only if $b$ is a unit.
\item $(\atm_f)^{-1}$ is reversible if and only if $a$ is a unit.
\item  $\atm_f$ is bireversible if and only if both  $a$ and $b$ are units.
\end{enumerate}
\end{theorem}
\begin{proof}
Inspection of \eqref{eq:transition.func} shows that $\atm_f$ is reversible if and only if, for each $\tilde s\in R$, the mapping
\[\alpha_{-sra}\mu_f\alpha_{sb}\longmapsto \alpha_{-(sb+\tilde{s})ra}\mu_f\alpha_{(sb+\tilde{s})b}\]
is one-to-one as this is the output function $\delta_{\tilde s}$ for the dual automaton $\partial \atm_f$ at state $\tilde s$.  Suppose first that $b$ is a unit and
\[\alpha_{-(sb+\tilde{s})ra}\mu_f\alpha_{(sb+\tilde{s})b}=\alpha_{-(s'b+\tilde{s})ra}\mu_f\alpha_{(s'b+\tilde{s})b}\]
for some $s,s'\in R$. By Proposition \ref{prop:conj}, this can be rewritten as
\[\alpha_{-(sb+\tilde{s})ra+(sb+\tilde{s})bf}\mu_f=\alpha_{-(s'b+\tilde{s})ra+(s'b+\tilde{r})bf}\mu_f.\]
Since $\mu_f$ is an invertible function this implies that
\[\alpha_{(sb+\tilde{s})(-ra+bf)}=\alpha_{(s'b+\tilde{s})(-ra+bf)}\]
and hence, by Proposition~\ref{prop:linind}, we get that
\[(sb+\tilde{s})b=(s'b+\tilde{s})b.\] Now since $b\in R^\times$, we deduce that $s=s'$. Therefore, $\partial \atm_f$ is invertible and $\atm_f$ is reversible.

Conversely, if $b$ is not a unit, then since $R$ is finite, $sb=0$ for some $s\neq 0$ in $S$.  Taking $\tilde s=0$, we then have $\alpha_{-sra}\mu_f\alpha_{sb}\mapsto \mu_f$ and $\mu_f\mapsto \mu_f$ under $\delta_{0}$.  Thus $\delta_0$ is not a permutation of the state set (as the states $\mu_f$ and $\alpha_{-sra}\mu_f\alpha_{sb}$ are distinct by Lemma~\ref{lem:numberofstates}) and so $\atm_f$ is not reversible.  This proves (i).

Now observe that $(\atm_f)^{-1}=\atm_{f^{-1}}$ and that \[f^{-1}=r^{-1}\left(\frac{1-bt}{1-at}\right).\] Therefore, by the same argument as above $(\atm_f)^{-1}$ is reversible if and only if $a$ is a unit, establishing (ii).  Item (iii) follows directly from (i) and (ii).
\end{proof}

Note that if $a-b$ is a unit, then we can identify the states of $\atm_f$ bijectively with $R$ via $s\mapsto \alpha_{-sra}\mu_f\alpha_{sb}$ by the above proof.  Under this identification the the edges of $\atm_f$ become
\begin{equation}\label{eq:transitions}
s\xrightarrow{\,\, \tilde s\,\mid\, r(\tilde s+(b-a)s)}sb+\tilde s.
\end{equation}

We now prove that when $a-b$ is a unit, the automaton group acts spherically transitively.  This is a necessary condition in order to perform spectral computations using the automaton group representation.

\begin{proposition}
Let \[f(t)=r\left(\frac{1-at}{1-bt}\right)\] where $r\in R^\times$ and $a,b\in R$. If $a-b$ is a unit of $R$, then $\mathbb{G}(\atm_f)$ is spherically transitive on $\tree_R$.  If, in addition, $r=1$ and $a=0$ or $b=0$, then $\mathbb{G}(\atm_f)$ is self-replicating.
\end{proposition}
\begin{proof}
The action of $\mathbb G(\atm_f)$ on level $d+1$ of $\tree_R$ can be identified with its action on $R\llbracket t\rrbracket/(t^{d+1})$ by affine mappings  for $d\geq 0$.  The proof of Theorem~\ref{thm:lamplighter} shows that $\mathbb{G}(\atm_f)$ contains the translations by $s(-ra+bf)f^m$ for all $s\in R$ and $m\in \Z$ or, equivalently by \eqref{equation:realbasis}, the $R$-span of all translations $\frac{1}{1-bt}f^m$ with $m\in \mathbb Z$ (since $r$ and $a-b$ are units).

Proposition~\ref{prop:linind} shows that the $|R|^{d+1}$ power series
\begin{equation}\label{eq:transitivity}
\frac{1}{1-bt}\left(c_0+c_1f+\cdots+c_df^d\right)
\end{equation}
 with $c_0,\ldots, c_d\in R$ are distinct. Multiplying \eqref{eq:transitivity} through by the invertible power series $(1-bt)^{d+1}$ yields that the $|R|^{d+1}$  polynomials
\begin{equation}\label{eq:reducednow}
c_0(1-bt)^d+c_1(1-at)(1-bt)^{d-1}+\cdots+c_d(1-at)^d
\end{equation}
of degree at most $d$  are all distinct.   Since \eqref{eq:reducednow} consists of polynomials of degree at most $d$, we obtain that the $|R|^{d+1}$ cosets in $R\llbracket t \rrbracket / (t^{d+1})$
 \begin{equation*}
c_0(1-bt)^d+c_1(1-at)(1-bt)^{d-1}+\cdots+c_d(1-at)^d+(t^{d+1})
\end{equation*}
are all distinct and hence multiplying by the unit $(1-bt)^{-(d+1)}$, we obtain that the $|R|^{d+1}$ cosets
\begin{equation*}
\frac{1}{1-bt}\left(c_0+c_1f+\cdots+c_df^d\right)+(t^{d+1})
\end{equation*}
are distinct.  But these are then all the cosets in $R\llbracket t\rrbracket/(t^{d+1})$ and so it follows that the action of $\mathbb G(\atm_f)$ on $R\llbracket t\rrbracket/(t^{d+1})$ contains all the translations and hence is transitive.  This completes the proof that $\mathbb G(\atm_f)$ is spherically transitive when $a-b$ is a unit.

Let $G=\mathbb G(\atm_f)$.
Suppose now that $r=1$ and $a$ or $b$ is zero.  Replacing $f$ by its inverse, we may assume that $b=0$.  Note that $a$ must then be a unit.  We already have that $G$ acts spherically transitively and so it suffices to show that $\{g|_0: g\in \Stab_G(0)\}=G$ by a standard argument.  Recall that by Proposition~\ref{prop:statesoff}, for $b=0$ and $r=1$, we have $G=\langle \alpha_{-sa}\mu_f : s\in R \rangle$. First note that $\mu_f\in \Stab_G(0)$ as $r=1$.  Also note that $(\mu_f)|_s=\alpha_{-sa}\mu_f$.   In particular, $\alpha_{-sa}\in G$ for all $s\in S$.  Since $a$ is a unit, for any $c\in R$, we can find $s\in R$ with $-sa=c$. Thus $\alpha_c\in G$ for all $c\in R$.  Then we have $\alpha_{-s}\mu_f\alpha_s\in \Stab_G(0)$ and $(\alpha_{-s}\mu_f\alpha_s)|_0 = (\mu_f)|_s=\alpha_{-as}\mu_f$ by Lemma~\ref{lem:states} for all $s\in R$. Thus $\{g|_0: g\in \Stab_G(0)\}=G$.  This completes the proof that $G$ is self-replicating.
\end{proof}

\section{Rings with $a$, $b$, and $a-b$ in $R^\times$}\label{sec:biringsclassification}
In this section, we characterize which finite abelian groups $A$ can be the additive group of a finite commutative ring with two units whose difference is a unit so that we may realize $A\wr \mathbb Z$ as a bireversible automaton group via Theorem~\ref{thm:lamplighter} and Theorem~\ref{thm:revbirev}.

We will write $(R, \mathfrak{m})$ to denote a finite commutative local ring with maximal ideal $\mathfrak{m}$. We remind the reader now of the following well-known proposition, whose proof we include for completeness.

\begin{proposition}\label{prop:localorder}
If $(R, \mathfrak{m})$ is a finite commutative local ring and $|R/\mathfrak{m}|=q$, then $|R|=q^i$ for some $i\geq 1$. In particular, $R$ has  prime power order.
\end{proposition}
\begin{proof}
Since $R$ is finite, hence Artinian, and $\mathfrak{m}$ is the Jacobson radical of $R$, it follows that $\mathfrak m$ is nilpotent and hence $\mathfrak m^{k+1}=0$ for some $k\geq 0$, which we take to be minimal.  Thus we have a filtration $R\supsetneq \mathfrak m\supsetneq \cdots \supsetneq \mathfrak m^{k+1}=0$ and $\mathfrak m^j/\mathfrak m^{j+1}$ is a finite dimensional vector space over the residue field $R/\mathfrak m$ for $0\leq j\leq k$.  Thus $|\mathfrak m^j/\mathfrak m^{j+1}|=q^{d_j}$ with $d_j=\dim \mathfrak m^j/\mathfrak m^{j+1}$ and hence $|R|=q^{d_1+\cdots +d_k}$ by repeated application of Lagrange's theorem.
\end{proof}

\begin{proposition}\label{prop:dividesa1}
Let $(R, \mathfrak{m})$ be a finite commutative local ring such that $\mathrm{char}(R/\mathfrak{m})=p$ and suppose $|R/\mathfrak{m}|=p^r$. If \[R^+\cong (\Z/p\Z)^{a_1} \oplus (\Z/{p^2}\Z)^{a_2} \oplus \cdots \oplus (\Z/{p^t}\Z)^{a_t},\] then $r| a_1$.
\end{proposition}
\begin{proof}
Consider the ideal $p^2R\subseteq \mathfrak{m}$. Then $(R/p^2R, \mathfrak{m}/p^2R)$ is a finite local commutative ring with residue field $R/\mathfrak{m}$ and \[(R/p^2R)^+\cong (\Z/p\Z)^{a_1} \oplus (\Z/{p^2}\Z)^{k}\] where $k=a_2+\cdots + a_t$. Note that it is possible that $k=0$. Without loss of generality, we may assume that $p^2R=0$ and $R^+\cong (\Z/p\Z)^{a_1} \oplus (\Z/{p^2}\Z)^{k}$. Then $|R|=p^{a_1+2k}=(p^r)^n$ for some $n$ by Proposition~\ref{prop:localorder}. Thus $r| a_1+2k$.

Let $I$ be the ideal $I=\{s\in R : ps=0\}$; it is proper since $R$ has characteristic $p^2$. Then $I^+\cong (\Z/p\Z)^{a_1}\oplus (\Z/p\Z)^k$ and $|I|=p^{a_1+k}$. Moreover, $I\subseteq \mathfrak{m}$ and $(R/I, \mathfrak{m}/I)$ is again a finite commutative local ring with residue field $R/\mathfrak{m}$. So $|R/I|=(p^r)^m$ for some $m$ by Proposition~\ref{prop:localorder}. Thus \[|I|=\frac{|R|}{|R/I|}=\frac{p^{rn}}{p^{rm}}=p^{r(n-m)}=p^{a_1+k}\] and so $r|a_1+k$.
We conclude that  $r| 2(a_1+k)-(a_1+2k)=a_1$.
\end{proof}

The following theorem is presumably well known, but we could not find a reference.

\begin{theorem}\label{thm:rdividesai}
Let $(R, \mathfrak{m})$ be a finite commutative local ring with $R/\mathfrak{m}$ of characteristic $p$ and $|R/\mathfrak{m}|=p^r$. Let \[R^+\cong (\Z/p\Z)^{a_1}\oplus (\Z/{p^2}\Z)^{a_2}\oplus \cdots \oplus (\Z/{p^t}\Z)^{a_t}.\] Then $r|a_i$ for $1\leq i \leq s$.
\end{theorem}
\begin{proof}
We already have that $r|a_1$ by Proposition \ref{prop:dividesa1}. Assume that, inductively, $r|a_1,\ldots, a_k$ with $1\leq k<t$. Let  $I$ be the ideal $\{s\in R : p^ks=0\}$. It's a proper ideal since $k<t$ and so $I$ is contained in $\mathfrak{m}$.  Thus $(R/I, \mathfrak{m}/I)$ is a finite commutative ring with residue field $R/\mathfrak{m}$. But $(R/I)^+\cong (\Z/p\Z)^{a_{k+1}}\oplus \cdots \oplus (\Z/{p^{t-k}}\Z)^{a_t}$ and so $r|a_{k+1}$ by Proposition \ref{prop:dividesa1}. The result follows by induction.
\end{proof}

Let $p$ be a prime and $m,r\geq 1$.  Then there is a unique (up to isomorphism) finite commutative ring  $R=GR(p^m,r)$ of characteristic $p^m$, called a \emph{Galois ring}, such that $|R|=p^{mr}$ and $R/pR\cong \mathbb F_{p^r}$.  It is a local ring with maximal ideal $pR$.   We do not prove the uniqueness of Galois rings, as we do not need it.  The interested reader is referred to~\cite{Bini02} for details.

\begin{theorem}\label{t:galois.ring}
Let $p$ be a prime and $m,r\geq 1$.  Then there is a finite commutative local ring $R=GR(p^m,r)$ with maximal ideal $pR$ such that $R/pR\cong \mathbb F_{p^r}$ and $R^+\cong (\mathbb Z/p^m\Z)^r$.
\end{theorem}
\begin{proof}
First we give an elementary construction.  Let $\alpha$ be a primitive element of $\mathbb F_{p^r}$, so that $\mathbb F_{p^r}=\mathbb F_p(\alpha)$.  Let $q(x)\in (\Z/p\Z)[x]$ be the minimal polynomial of $\alpha$; it is a monic polynomial of degree $r$.  We can choose a monic polynomial $Q(x)\in (\Z/p^m\Z)[x]$ of degree $r$ which reduces to $q(x)$ by simply identifying the coefficients of $q(x)$ with integers between $0$ and $p-1$.  Set $R=(\Z/p^m\Z)[x]/(Q(x))$.  First note that since $Q(x)$ is monic, it follows that $(Q(x))$ consists of polynomials of degree at least $r$.  Thus the cosets $1+(Q(x)), x+(Q(x)),\ldots, x^{r-1}+(Q(x))$ are linearly independent over $\Z/p^m\Z$.  Also, as $Q(x)$ is monic, $x^r+(Q(x))$ is in the $\Z/p^m\Z$-span of the cosets $x^j+(Q(x))$  with $0\leq j\leq r-1$ and hence, by induction, so are all cosets $x^t+(Q(x))$ with $t\geq r$.  We conclude that $R^+\cong (\Z/p^m\Z)^r$.  Also $pR$ is a nilpotent ideal of $R$, as $(pR)^m=p^mR=0$, and so $pR$ is contained in every maximal ideal of $R$. But $R/pR\cong (\Z/p\Z)[x]/(q(x))\cong \mathbb F_{p^r}$ is a field, and so $pR$ is a maximal ideal and hence the unique maximal ideal of $R$.  This completes the proof.

We now give a more conceptual construction.
Let $\Q_p$ be the $p$-adic rationals and $\Z_p$ the $p$-adic integers. Take $\zeta$ to be a primitive $p^r-1$ root of unity in an algebraic closure of $\Q_p$.  Then  $\Q_p(\zeta)$ is the unique unramified extension of $\Q_p$ of degree $r$.  The ring of integers in $\Q_p(\zeta)$ is $\mathcal{O}=\Z_p[\zeta]$. Moreover, $\mathcal O$ is a complete discrete valuation ring with maximal ideal $p\mathcal{O}$ and residue field $\mathcal{O}/p\mathcal{O}=\mathbb{F}_{p^r}$.  Details can be found in~\cite[Chpt.~III, Thm.~25]{FrT93} and~\cite[Equation (3.3), Page~136]{FrT93}.
As a $\Z_p$-module, $\mathcal{O}$ is isomorphic to $(\Z_p)^r$ and thus $R=\mathcal{O}/p^m\mathcal{O}$ is a finite commutative local ring with additive group isomorphic to $(\Z/p^m\Z)^r$ and maximal ideal $pR$ with $R/pR\cong \mathbb F_{p^r}$.
\end{proof}

Finally, we are ready to classify which finite abelian groups can be the additive group of a ring with two units whose difference is a unit.

\begin{theorem}
Let $A$ be a finite abelian group. Then there is a finite commutative ring $R$ with $R^+\cong A$ and two elements $a,b\in R^\times$ with $a-b\in R^\times$ if and only if $A\cong A_1 \oplus A_2$ where $A_1$ has odd order and $A_2\cong (\Z/2\Z)^{a_1}\oplus(\Z/{2^2}\Z)^{a_2}\oplus \cdots\oplus (\Z/{2^t}\Z)^{a_t}$ with $a_i\neq 1$ for all $1\leq i \leq t$.
\end{theorem}
\begin{proof}
Note that if the ring $R$ decomposes as $R\cong R_1\times \cdots \times R_n$, then there exist $a$ and $b$ in $R^\times$ with $a-b$ also in $R^\times$ if and only if there exist $a_i$ and $b_i$ in $R_i^\times$ with $a_i-b_i$ in $R^\times$ for all $1\leq i \leq n$. Therefore, to show existence it suffices to find $\Z/{p^k}\Z$ for an odd prime $p$ and $(\Z/{2^m}\Z)^r$ for $r\geq 2$ and $m\geq 1$ as the additive group of a ring with this property.

For the odd prime case, $\Z/{p^k}\Z$, considered as a ring in the standard way, already serves our purpose since $1,-1$ and $2=1-(-1)$ are all units in this ring.

For the even case, % let $\Q_2$ be the $2$-adic rationals and $\Z_2$ the $2$-adic integers. Take $\zeta$ to be a primitive $2^r-1$ root of unity in an algebraic closure of $\Q_2$.  Then $\Q_2$ has a unique unramified extension of degree $r$, namely $\Q_2(\zeta)$.  Now the ring of integers in $\Q_2$ is $\mathcal{O}=\Z_2[\zeta]$. Moreover, $\mathcal O$ is a complete discrete valuation ring with maximal ideal $2\mathcal{O}$ and residue field $\mathcal{O}/2\mathcal{O}=\mathbb{F}_{2^r}$.
%See~\cite[Chpt.~III, Thm.~25 and Equation (3.3), Page~136]{FrT93} for details.
%As a $\Z_2$-module, $\mathcal{O}$ is isomorphic to $(\Z_2)^r$ and thus $R=\mathcal{O}/2^m\mathcal{O}$ is a finite commutative local ring with additive %group isomorphic to $(\Z/2^m\Z)^r$ and maximal ideal $2R$. Now, if $r>1$, then $|\mathcal{O}/2\mathcal{O}|=|\mathbb{F}_{2^r}|>3$ and we can choose $\bar{a}$ and $\bar{b}$ in $\mathcal{O}/ 2\mathcal{O}$ with $\bar{a}-\bar{b}\neq 0$. Let $a$ and $b$ be pre-images of $\bar{a}$ and $\bar{b}$ in $R$. Then $a$, $b$, and $a-b$ are all units of $R$.
for $r\geq 2$, let $R=GR(2^m,r)$ as per Theorem~\ref{t:galois.ring}.  Then $R^+\cong (\Z/2^m\Z)^r$, $R$ is local with maximal ideal $2R$ and $R/2R\cong \mathbb F_{2^r}$.   So we can choose $\bar{a}$ and $\bar{b}$ non-zero in $R/2R$ with $\bar{a}-\bar{b}\neq 0$. Let $a$ and $b$ be pre-images of $\bar{a}$ and $\bar{b}$ in $R$. Then $a$, $b$, and $a-b$ are all units of $R$.
%for $r\geq 2$ let $R=GR(2^m, r)$ be the Galois ring of order $2^{mr}$ of rank $r$. It is the unique Galois extension of $\Z_{2^m}$ of rank $r$ and is a commutative local ring with residue field $\mathbb{F}_{2^r}$ and additive group $\Z_{2^m}^r$. As it is a Galois extension of $\Z_{2^m}$, it is a free $\Z_{2^m}$-module and its unique maximal ideal is $2R$. Moreover $R\cong \Z_{2^m}[t]/(f(t))$ where $f(t)$ is a monic, irreducible polynomial which is not a zero divisor and is obtained as Hensel lift of a primitive polynomial over $\mathbb{F}_2$ defining $\mathbb{F}_{2^r}$.  Since $|\mathbb{F}_{2^r}^\times|\geq 3$ we can find $a, b\in R\backslash 2R$ such that $a+pR\neq b+2R$. Then $a,b \in R^\times$ and $a-b\notin 2R$ so $a-b\in R^\times$. This proves existence.

We now show that no other finite abelian groups can be the additive group of a ring with two units whose difference is a unit. Suppose $A=A_1\oplus A_2$ with $A_1$ and $A_2$ as in the statement of the theorem. Suppose $R$ is a ring with $A\cong R^+$. Write $R=R_1\times \cdots \times R_n$ with $R_i$ local for $1\leq i \leq n$. If $a_i=1$ for some $i$, then there exists an $R_j$ of even order with $\Z/2^i\Z$ as a direct summand in $R_j^+$ with multiplicity $1$. Then by Theorem \ref{thm:rdividesai}, the residue field of $R_j$ must be $\mathbb{F}_2$. Let $\mathfrak{m}_j$ be the maximal ideal of $R_j$. If $a=(a_1,  \ldots, a_n)$ and $b=(b_1, \ldots, b_n)$ are in $R^\times$ then $a_j+\mathfrak{m}_j=b_j+\mathfrak{m}_j$ as the residue field of $R_j$ is $\mathbb{F}_2$. Thus $a_j-b_j\in \mathfrak{m}_j$ and is not a unit. Consequently, $a-b$ is not a unit. So $R$ does not have two units whose difference is a unit.
\end{proof}

%We remark that the ring $R=\mathcal O/2^m\mathcal O$ constructed above is also known as the Galois ring $GR(2^m,r)$.  See~\cite{Bini02} for a construction avoiding $2$-adics.

\begin{corollary}
Let $A\cong A_1 \oplus A_2$ be a finite abelian group where $A_1$ has odd order and $A_2\cong (\Z/2\Z)^{a_1}\oplus(\Z/{2^2}\Z)^{a_2}\oplus \cdots\oplus (\Z/{2^t}\Z)^{a_t}$ with $a_i\neq 1$ for all $1\leq i \leq t$.  Then there is a bireversible automaton over an $|A|$-element alphabet with $|A|$ states generating a group isomorphic to $A\wr \mathbb Z$.
\end{corollary}

\begin{remark}
Note that if $n$ is even, then there is no finite commutative ring $R$ with $R^+\cong \mathbb Z/n\mathbb Z$ containing two units $a,b$ with $a-b$ a unit.  So we cannot realize the lamplighter group $(\mathbb Z/n\mathbb Z)\wr \mathbb Z$ using bireversible automata via our methods when $n$ is even.
\end{remark}

\section{Examples}
Recall that for a fixed $r$, $a$, and $b$ in our finite commutative ring $R$, the states of $\atm_f$, for \[f=r\left(\frac{1-at}{1-bt}\right),\] are $\{\alpha_{-sra}\mu_f\alpha_{sb} \mid s\in R\}$. If $a-b$ is a unit of $R$, then different $s\in R$ yield different states by Lemma~\ref{lem:numberofstates}. For this reason, in all of the following figures and tables we use $s$ to denote the state $\alpha_{-sra}\mu_f\alpha_{sb}$.  With this notation, transitions are given as in \eqref{eq:transitions}.

\begin{example}
As a first example we show that the power series method can be used to recreate the bireversible automaton given in \cite{bdr16} for the group $\Z/3\Z \wr \Z$.  This was already observed by Bondarenko and Savchuk in their work on rational series over fields. We use $f(t)=2\left(\frac{1-2t}{1-t}\right)$ in $(\Z/3\Z)\llbracket t \rrbracket$. With $r=2, a=2$, and $b=1$, we see that the states of $\atm_f$ are of the form $\alpha_{-s}\mu_f\alpha_s$. The output function is given by $\lambda(s, \tilde{s})=-s+2(s+\tilde{s})$ and the transition function is given by $\delta(s, \tilde{s})=\alpha_{-(s+\tilde{s})}\mu_f \alpha_{s +\tilde{s}}$ for any $\tilde{s}\in \Z/3\Z$. As was observed by Bondarenko, D'Angeli, and Rodaro, this automaton is equivalent to its dual. See Figure~\ref{fig:BDR}.

\begin{figure}[htbp]
\begin{center}
\begin{tikzpicture}[->,shorten >=1pt, auto, node distance=5.5cm,semithick,
inner sep=1pt, bend angle=27,scale=.5,font=\footnotesize]
\tikzset{every state/.style={minimum size=0.75cm}}.
\node [state]   (A) {$0$};
\node [state]   (C) [above right of=A] {$2$};
\node [state]   (B) [below right of=C] {$1$};

\path  (A) edge [in=180, out=240, loop] node [below left] {$0\mid 0$} (A)
(A) edge [bend right]  node [below]  {$1 \mid 2$} (B)
(A) edge  node [below right]  {$2 \mid 1$} (C)
(B) edge [in=300, out=0, loop] node [below right] {$0\mid 1$} (B)
(B) edge node [above]  {$2 \mid 2$} (A)
(B) edge [bend right]  node [above right]  {$1 \mid 0$} (C)
(C) edge [in=60, out=120, loop] node [above] {$0\mid 2$} (C)
(C) edge [bend right]  node [above left]  {$1 \mid 1$} (A)
(C) edge node [below left]  {$2 \mid 0$} (B);
\end{tikzpicture}
\end{center}
\caption{The bireversible automata given by Bodarenko, D'Angeli, Rodaro for $\Z/3\Z\wr \Z$ and corresponding to $f(t)=2\left(\dfrac{1-2t}{1-t}\right)$.\label{fig:BDR}}
\end{figure}
\end{example}

\begin{example}
Let $R=\Z/6\Z$ with the standard ring structure and take $a=3$, $b=2$, and $r=1$ so that neither $a$ nor $b$ is a unit. Therefore, for $f=\dfrac{1-3t}{1-2t}$, both $\atm_f$ and $\atm_{f^{-1}}$ are not reversible by Theorem~\ref{thm:revbirev}. Here, the states are given by $\alpha_{3s}\mu_f\alpha_{2s}$. The transition function is given by
\[\delta(\alpha_{3s}\mu_f\alpha_{2s}, \tilde{s})=\alpha_{3(2s+\tilde{s})}\mu_f\alpha_{2(2s+\tilde{s})}=\alpha_{3\tilde{s}}\mu_f\alpha_{4s+\tilde{s}}\]
 and the output function is given by
 \[\lambda(\alpha_{3s}\mu_f\alpha_{2s}, \tilde{s})=\tilde{s}+5s=\tilde{s}-s.\] It is straightforward to check that $\delta_x(\alpha_{3s}\mu_f\alpha_{2s})=\delta_x(\alpha_{3s'}\mu_f\alpha_{2s'})$ whenever $s \equiv s' \bmod 3$, which verifies that $\atm_f$ is not reversible. See Figure~\ref{fig:strongnotreversible}.
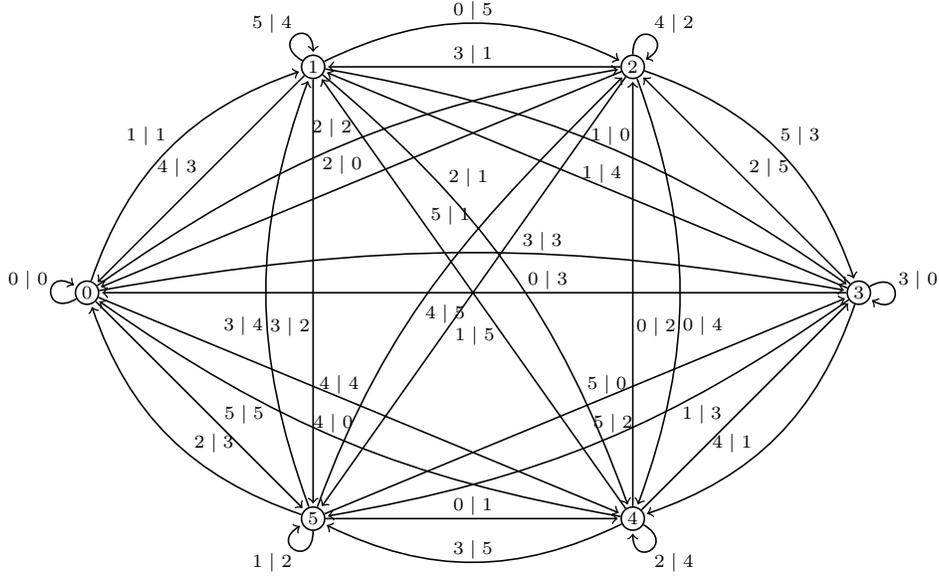
\begin{figure}[htbp]
\begin{center}
\begin{tikzpicture}[->,shorten >=1pt, auto, node distance=4.25cm,semithick,
inner sep=1pt, bend angle=26,font=\tiny,scale=.5]
\tikzset{every state/.style={minimum size=0}}.
\node [state]   (A) {$0$};
\node [state]   (B) [above right of=A] {$1$};
\node [state]   (C) [right of=B] {$2$};
\node [state]   (D) [below right of=C] {$3$};
\node [state]   (E) [below left of=D] {$4$};
\node [state]   (F) [left of=E] {$5$};
\path
%1skipcycle
(A) edge [bend left]  node [above left]  {$1 \mid 1$} (B)
(B) edge node [above left]  {$4 \mid 3$} (A)
(B) edge [bend left]  node [above]  {$0 \mid 5$} (C)
(C) edge node [above]  {$3 \mid 1$} (B)
(C) edge [bend left]  node [above right]  {$5 \mid 3$} (D)
(D) edge node [above right]  {$2 \mid 5$} (C)
(D) edge [bend left]  node [pos=.6, above left]  {$4 \mid 1$} (E)
(E) edge node [pos=.4, above left]  {$1 \mid 3$} (D)
(E) edge [bend left]  node [above]  {$3 \mid 5$} (F)
(F) edge node [above]  {$0 \mid 1$} (E)
(F) edge [bend left]  node [pos=.4, above right]  {$2 \mid 3$} (A)
(A) edge node [pos=.6, above right]  {$5 \mid 5$} (F)
%2skipcycle
(A) edge [out=35, in=190] node [pos=.52,above left]  {$2 \mid 2$} (C)
(C) edge node [pos=.49, above left]  {$2 \mid 0$} (A)
(B) edge [out=350, in=145] node [pos=.52, above]  {$1 \mid 0$} (D)
(D) edge node [pos=.47, above]  {$1 \mid 4$} (B)
(C) edge [out=290, in=70] node [pos=.55, below right]  {$0 \mid 4$} (E)
(E) edge node [pos=.45, below right]  {$0 \mid 2$} (C)
(D) edge [out=215, in=10] node [pos=.475, above]  {$5 \mid 2$} (F)
(F) edge node [pos=.54, above]  {$5 \mid 0$} (D)
(E) edge [out=170, in=325] node [pos=.525, above]  {$4 \mid 0$} (A)
(A) edge node [pos=.46, above]  {$4 \mid 4$} (E)
(F) edge [out=110, in=250] node [pos=.45, below left]  {$3 \mid 4$} (B)
(B) edge node [pos=.55, below left]  {$3 \mid 2$} (F)
%3skipcycle
(A) edge [out=10, in=170] node [pos=.6, above]  {$3 \mid 3$} (D)
(D) edge node [pos=.4, above]  {$0 \mid 3$} (A)
(B) edge [out=320, in=110] node [pos=.3, above right]  {$2 \mid 1$} (E)
(E) edge node [pos=.65, above right]  {$5 \mid 1$} (B)
(C) edge node [pos=.57, below right]  {$1 \mid 5$} (F)
(F) edge [out=70, in=225] node [pos=.43, below right]  {$4 \mid 5$} (C)
%loops
(A) edge [in=150, out=210, loop] node [above left] {$0\mid 0$} (A)
(B) edge [in=90, out=150, loop] node [above left] {$5\mid 4$} (B)
(C) edge [in=30, out=90, loop] node [above right] {$4\mid 2$} (C)
(D) edge [in=330, out=30, loop] node [above right] {$3\mid 0$} (D)
(E) edge [in=270, out=330, loop] node [below right] {$2\mid 4$} (E)
(F) edge [in=210, out=270, loop] node [below left] {$1\mid 2$} (F);
\end{tikzpicture}
\end{center}
\caption{An automaton which generates $\Z/6\Z\wr\Z$ that is not reversible and whose inverse is also not reversible.\label{fig:strongnotreversible}}
\end{figure}

\end{example}

\begin{example}
Our next example is that of a bireversible automaton that generates $A \wr \Z$ where $A$ is the additive group of a ring but not of a field. We take $A=\Z/9\Z$ endowed with the standard ring structure. Taking $r=2$, $a=1$, and $b=2$, we have that $a-b=-1$, and so $a$, $b$, and $a-b$ are all units. For these values of $r$, $a$, and $b$, the states of $\atm_f$ are $\alpha_{-2s}\mu_f\alpha_{2s}$. The transition and output functions are given by
\[\delta(\alpha_{-2s}\mu_f\alpha_{2s}, \tilde{s})=\alpha_{-2(2s+\tilde{s})}\mu_f\alpha_{2(2s+\tilde{s})}\]
and
\[\lambda(\alpha_{-2s}\mu_f\alpha_{2s}, \tilde{s})=2(s+\tilde{s}).\]

Rather than drawing the $\atm_f$ for $f=2\left(\dfrac{1-t}{1-2t}\right)$ which has $9$ states and a $9$ letter alphabet, we describe the transition and output functions using Tables~\ref{tab:nonfieldbirevtrans} and~\ref{tab:nonfieldbirevout}.

\begin{table}[htbp]
\scalebox{0.8}{
\begin{tabular}{|c|| c | c | c | c | c | c | c | c | c|}
\hline
state $\backslash$ letter & 0 & 1 & 2 & 3 & 4 & 5 & 6 & 7 & 8\\
\hline \hline
0 & 0 & 1 & 2 & 3 & 4 & 5 & 6 & 7 & 8\\
\hline
1 & 2 & 3 & 4 & 5 & 6 & 7 & 8 & 0 & 1\\
\hline
2 & 4 & 5 & 6 & 7 & 8 & 0 & 1 & 2 & 3\\
\hline
3 & 6 & 7 & 8 & 0 & 1 & 2 & 3 & 4 & 5\\
\hline
4 & 8 & 0 & 1 & 2 & 3 & 4 & 5 & 6 & 7\\
\hline
5 & 1 & 2 & 3 & 4 & 5 & 6 & 7 & 8 & 0\\
\hline
6 & 3 & 4 & 5 & 6 & 7 & 8 & 0 & 1 & 2\\
\hline
7  & 5 & 6 & 7 & 8 & 0 & 1 & 2 & 3 & 4\\
\hline
8 & 7 & 8 & 0 & 1 & 2 & 3 & 4 & 5 & 6\\
\hline
\end{tabular}}
\caption{The transition table for $f=2\left(\dfrac{1-t}{1-2t}\right)$ over $\Z/9\Z$. \label{tab:nonfieldbirevtrans}}
\end{table}

\begin{table}[htbp]
\scalebox{0.8}{
\begin{tabular}{|c|| c | c | c | c | c | c | c | c | c|}
\hline
state $\backslash$ letter & 0 & 1 & 2 & 3 & 4 & 5 & 6 & 7 & 8\\
\hline \hline
0 & 0 & 2 & 4 & 6 & 8 & 1 & 3 & 5 & 7\\
\hline
1 & 2 & 4 & 6 & 8 & 1 & 3 & 5 & 7 & 0\\
\hline
2 & 4 & 6 & 8 & 1 & 3 & 5 & 7 & 0 & 2\\
\hline
3 & 6 & 8 & 1 & 3 & 5 & 7 & 0 & 2 & 4\\
\hline
4 & 8 & 1 & 3 & 5 & 7 & 0 & 2 & 4 & 6\\
\hline
5 & 1 & 3 & 5 & 7 & 0 & 2 & 4 & 6 & 8\\
\hline
6 & 3 & 5 & 7 & 0 & 2 & 4 & 6 & 8 & 1\\
\hline
7  & 5 & 7 & 0 & 2 & 4 & 6 & 8 & 1 & 3\\
\hline
8 & 7 & 0 & 2 & 4 & 6 & 8 & 1 & 3 & 5\\
\hline
\end{tabular}}
\caption{The output table for $f=2\left(\dfrac{1-t}{1-2t}\right)$ over $\Z/9\Z$. \label{tab:nonfieldbirevout}}
\end{table}

\end{example}

\begin{example}
As a final example, we construct a bireversible automaton generating $(\Z/4\Z)^2\wr \Z$ using the ring described in Theorem~\ref{t:galois.ring}. In other words, we take $\mathcal{O}=\Z_2[\zeta]$ with $\zeta$ a third root of unity and $R=\mathcal{O}/4\mathcal{O}\cong \Z/4\Z[\zeta]$. Using the isomorphism $R \cong (\Z/4\Z[x])/(1+x+x^2)$ and taking $r=1$, $a=1$, and $b=2+\zeta$, we find that $a$, $b$, and $a-b$ are all units with inverses $1$, $3+\zeta$, and $\zeta$ respectively. With this choice of $r$, $a$, and $b$, the transition and output tables are given by Tables~\ref{tab:galringtrans} and~\ref{tab:galringout} respectively for $f=r\left(\dfrac{1-at}{1-bt}\right)$.

\begin{landscape}
\begin{table}[htbp]
\scalebox{0.8}{
\begin{tabular}{|c|| c | c | c | c | c | c | c | c | c | c | c | c |  c | c | c | c |}
\hline
state $\backslash$ letter & $0$ & $1$ & $2$ & $3$ & $\zeta$ & $1+\zeta$ & $2+\zeta$ & $3+\zeta$ & $2\zeta$ & $1+2\zeta$ & $2+2\zeta$ & $3+2\zeta$ & $3\zeta$ & $1+3\zeta$ & $2+3\zeta$ & $3+3\zeta$\\
\hline \hline
$0$ & $0$ & $1$ & $2$ & $3$ & $\zeta$ & $1+\zeta$ & $2+\zeta$ & $3+\zeta$ & $2\zeta$ & $1+2\zeta$ & $2+2\zeta$ & $3+2\zeta$ & $3\zeta$ & $1+3\zeta$ & $2+3\zeta$ & $3+3\zeta$\\
\hline
$1$ & $2+\zeta$ & $3+\zeta$ & $\zeta$ & $1+\zeta$ & $2+2\zeta$ & $3+2\zeta$ & $2\zeta$ & $1+2\zeta$ & $2+3\zeta$ & $3+3\zeta$ & $3\zeta$ & $1+3\zeta$ & $2$ & $3$ & $0$ & $1$\\
\hline
$2$ & $2\zeta$ & $1+2\zeta$ & $2+2\zeta$ & $3+2\zeta$ & $3\zeta$ & $1+3\zeta$ & $2+3\zeta$ & $3+3\zeta$ & $0$ & $1$ & $2$ & $3$ & $\zeta$ & $1+\zeta$ & $2+\zeta$ & $3+\zeta$\\
\hline
$3$ & $2+3\zeta$ & $3+3\zeta$ & $3\zeta$ & $1+3\zeta$ & $2$ & $3$ & $0$ & $1$ & $2+\zeta$ & $3+\zeta$ & $\zeta$ & $1+\zeta$ & $2+2\zeta$ & $3+2\zeta$ & $2\zeta$ & $1+2\zeta$\\
\hline
$\zeta$ & $3+\zeta$ & $\zeta$ & $1+\zeta$ & $2+\zeta$ & $3+2\zeta$ & $2\zeta$ & $1+2\zeta$ & $2+2\zeta$ & $3+3\zeta$ & $3\zeta$ & $1+3\zeta$ & $2+3\zeta$ & $3$ & $0$ & $1$ & $2$\\
\hline
$1+\zeta$ & $1+2\zeta$ & $2+2\zeta$ & $3+2\zeta$ & $2\zeta$ & $1+3\zeta$ & $2+3\zeta$ & $3+3\zeta$ & $3\zeta$ & $1$ & $2$ & $3$ & $0$ & $1+\zeta$ & $2+\zeta$ & $3+\zeta$ & $\zeta$\\
\hline
$2+\zeta$ & $3+3\zeta$ & $3\zeta$ & $1+3\zeta$ & $2+3\zeta$ & $3$ & $0$ & $1$ & $2$ & $3+\zeta$ & $\zeta$ & $1+\zeta$ & $2+\zeta$ & $3+2\zeta$ & $2\zeta$ & $1+2\zeta$ & $2+2\zeta$\\
\hline
$3+\zeta$ & $1$ & $2$ & $3$ & $0$ & $1+\zeta$ & $2+\zeta$ & $3+\zeta$ & $\zeta$ & $1+2\zeta$ & $2+2\zeta$ & $3+2\zeta$ & $2\zeta$ & $1+3\zeta$ & $2+3\zeta$ & $3+3\zeta$ & $3\zeta$\\
\hline
$2\zeta$ & $2+2\zeta$ & $3+2\zeta$ & $2\zeta$ & $1+2\zeta$ & $2+3\zeta$ & $3+3\zeta$ & $3\zeta$ & $1+3\zeta$ & $2$ & $3$ & $0$ & $1$ & $2+\zeta$ & $3+\zeta$ & $\zeta$ & $1+\zeta$\\
\hline
$1+2\zeta$ & $3\zeta$ & $1+3\zeta$ & $2+3\zeta$ & $3+3\zeta$ & $0$ & $1$ & $2$ & $3$ & $\zeta$ & $1+\zeta$ & $2+\zeta$ & $3+\zeta$ & $2\zeta$ & $1+2\zeta$ & $2+2\zeta$ & $3+2\zeta$\\
\hline
$2+2\zeta$ & $2$ & $3$ & $0$ & $1$ & $2+\zeta$ & $3+\zeta$ & $\zeta$ & $1+\zeta$ & $2+2\zeta$ & $3+2\zeta$ & $2\zeta$ & $1+2\zeta$ & $2+3\zeta$ & $3+3\zeta$ & $3\zeta$ & $1+3\zeta$\\
\hline
$3+2\zeta$ & $\zeta$ & $1+\zeta$ & $2+\zeta$ & $3+\zeta$ & $2\zeta$ & $1+2\zeta$ & $2+2\zeta$ & $3+2\zeta$ & $3\zeta$ & $1+3\zeta$ & $2+3\zeta$ & $3+3\zeta$ & $0$ & $1$ & $2$ & $3$\\
\hline
$3\zeta$ & $1+3\zeta$ & $2+3\zeta$ & $3+3\zeta$ & $3\zeta$ & $1$ & $2$ & $3$ & $0$ & $1+\zeta$ & $2+\zeta$ & $3+\zeta$ & $\zeta$ & $1+2\zeta$ & $2+2\zeta$ & $3+2\zeta$ & $2\zeta$\\
\hline
$1+3\zeta$ & $3$ & $0$ & $1$ & $2$ & $3+\zeta$ & $\zeta$ & $1+\zeta$ & $2+\zeta$ & $3+2\zeta$ & $2\zeta$ & $1+2\zeta$ & $2+2\zeta$ & $3+3\zeta$ & $3\zeta$ & $1+3\zeta$ & $2+3\zeta$\\
\hline
$2+3\zeta$ & $1+\zeta$ & $2+\zeta$ & $3+\zeta$ & $\zeta$ & $1+2\zeta$ & $2+2\zeta$ & $3+2\zeta$ & $2\zeta$ & $1+3\zeta$ & $2+3\zeta$ & $3+3\zeta$ & $3\zeta$ & $1$ & $2$ & $3$ & $0$\\
\hline
$3+3\zeta$ & $3+2\zeta$ & $2\zeta$ & $1+2\zeta$ & $2+2\zeta$ & $3+3\zeta$ & $3\zeta$ & $1+3\zeta$ & $2+3\zeta$ & $3$ & $0$ & $1$ & $2$ & $3+\zeta$ & $\zeta$ & $1+\zeta$ & $2+\zeta$\\
\hline

\end{tabular}}
\caption{The transition table for $f=r\left(\dfrac{1-at}{1-bt}\right)$ with $r=1$, $a=1$, and $b=2+\zeta$.}\label{tab:galringtrans}
\end{table}
\end{landscape}

\begin{landscape}
\begin{table}[htbp]
\scalebox{0.8}{
\begin{tabular}{|c|| c | c | c | c | c | c | c | c | c | c | c | c |  c | c | c | c |}
\hline
state $\backslash$ letter & $0$ & $1$ & $2$ & $3$ & $\zeta$ & $1+\zeta$ & $2+\zeta$ & $3+\zeta$ & $2\zeta$ & $1+2\zeta$ & $2+2\zeta$ & $3+2\zeta$ & $3\zeta$ & $1+3\zeta$ & $2+3\zeta$ & $3+3\zeta$\\
\hline \hline
$0$ & $0$ & $1$ & $2$ & $3$ & $\zeta$ & $1+\zeta$ & $2+\zeta$ & $3+\zeta$ & $2\zeta$ & $1+2\zeta$ & $2+2\zeta$ & $3+2\zeta$ & $3\zeta$ & $1+3\zeta$ & $2+3\zeta$ & $3+3\zeta$\\
\hline
$1$ & $1+\zeta$ & $2+\zeta$ & $3+\zeta$ & $\zeta$ & $1+2\zeta$ & $2+2\zeta$ & $3+2\zeta$ & $2\zeta$ & $1+3\zeta$ & $2+3\zeta$ & $3+3\zeta$ & $3\zeta$ & $1$ & $2$ & $3$ & $0$\\
\hline
$2$ & $2+2\zeta$ & $3+2\zeta$ & $2\zeta$ & $1+2\zeta$ & $2+3\zeta$ & $3+3\zeta$ & $3\zeta$ & $1+3\zeta$ & $2$ & $3$ & $0$ & $1$ & $2+\zeta$ & $3+\zeta$ & $\zeta$ & $1+\zeta$\\
\hline
$3$ & $3+3\zeta$ & $3\zeta$ & $1+3\zeta$ & $2+3\zeta$ & $3$ & $0$ & $1$ & $2$ & $3+\zeta$ & $\zeta$ & $1+\zeta$ & $2+\zeta$ & $3+2\zeta$ & $2\zeta$ & $1+2\zeta$ & $2+2\zeta$\\
\hline
$\zeta$ & $3$ & $0$ & $1$ & $2$ & $3+\zeta$ & $\zeta$ & $1+\zeta$ & $2+\zeta$ & $3+2\zeta$ & $2\zeta$ & $1+2\zeta$ & $2+2\zeta$ & $3+3\zeta$ & $3\zeta$ & $1+3\zeta$ & $2+3\zeta$\\
\hline
$1+\zeta$ & $\zeta$ & $1+\zeta$ & $2+\zeta$ & $3+\zeta$ & $2\zeta$ & $1+2\zeta$ & $2+2\zeta$ & $3+2\zeta$ & $3\zeta$ & $1+3\zeta$ & $2+3\zeta$ & $3+3\zeta$ & $0$ & $1$ & $2$ & $3$\\
\hline
$2+\zeta$ & $1+2\zeta$ & $2+2\zeta$ & $3+2\zeta$ & $2\zeta$ & $1+3\zeta$ & $2+3\zeta$ & $3+3\zeta$ & $3\zeta$ & $1$ & $2$ & $3$ & $0$ & $1+\zeta$ & $2+\zeta$ & $3+\zeta$ & $\zeta$\\
\hline
$3+\zeta$ & $2+3\zeta$ & $3+3\zeta$ & $3\zeta$ & $1+3\zeta$ & $2$ & $3$ & $0$ & $1$ & $2+\zeta$ & $3+\zeta$ & $\zeta$ & $1+\zeta$ & $2+2\zeta$ & $3+2\zeta$ & $2\zeta$ & $1+2\zeta$\\
\hline
$2\zeta$ & $2$ & $3$ & $0$ & $1$ & $2+\zeta$ & $3+\zeta$ & $\zeta$ & $1+\zeta$ & $2+2\zeta$ & $3+2\zeta$ & $2\zeta$ & $1+2\zeta$ & $2+3\zeta$ & $3+3\zeta$ & $3\zeta$ & $1+3\zeta$\\
\hline
$1+2\zeta$ & $3+\zeta$ & $\zeta$ & $1+\zeta$ & $2+\zeta$ & $3+2\zeta$ & $2\zeta$ & $1+2\zeta$ & $2+2\zeta$ & $3+3\zeta$ & $3\zeta$ & $1+3\zeta$ & $2+3\zeta$ & $3$ & $0$ & $1$ & $2$\\
\hline
$2+2\zeta$ & $2\zeta$ & $1+2\zeta$ & $2+2\zeta$ & $3+2\zeta$ & $3\zeta$ & $1+3\zeta$ & $2+3\zeta$ & $3+3\zeta$ & $0$ & $1$ & $2$ & $3$ & $\zeta$ & $1+\zeta$ & $2+\zeta$ & $3+\zeta$\\
\hline
$3+2\zeta$ & $1+3\zeta$ & $2+3\zeta$ & $3+3\zeta$ & $3\zeta$ & $1$ & $2$ & $3$ & $0$ & $1+\zeta$ & $2+\zeta$ & $3+\zeta$ & $\zeta$ & $1+2\zeta$ & $2+2\zeta$ & $3+2\zeta$ & $2\zeta$\\
\hline
$3\zeta$ & $1$ & $2$ & $3$ & $0$ & $1+\zeta$ & $2+\zeta$ & $3+\zeta$ & $\zeta$ & $1+2\zeta$ &  $2+2\zeta$ & $3+2\zeta$ & $2\zeta$ & $1+3\zeta$ & $2+3\zeta$ & $3+3\zeta$ & $3\zeta$\\
\hline
$1+3\zeta$ & $2+\zeta$ & $3+\zeta$ & $\zeta$ & $1+\zeta$ & $2+2\zeta$ & $3+2\zeta$ & $2\zeta$ & $1+2\zeta$ & $2+3\zeta$ & $3+3\zeta$ & $3\zeta$ & $1+3\zeta$ & $2$ & $3$ & $0$ & $1$\\
\hline
$2+3\zeta$ & $3+2\zeta$ & $2\zeta$ & $1+2\zeta$ & $2+2\zeta$ & $3+3\zeta$ & $3\zeta$ & $1+3\zeta$ & $2+3\zeta$ & $3$ & $0$ & $1$ & $2$ & $3+\zeta$ & $\zeta$ & $1+\zeta$ & $2+\zeta$\\
\hline
$3+3\zeta$ & $3\zeta$ & $1+3\zeta$ & $2+3\zeta$ & $3+3\zeta$ & $0$ & $1$ & $2$ & $3$ & $\zeta$ & $1+\zeta$ & $2+\zeta$ & $3+\zeta$ & $2\zeta$ & $1+2\zeta$ & $2+2\zeta$ & $3+2\zeta$\\
\hline

\end{tabular}}
\caption{The output table for $f=r\left(\dfrac{1-at}{1-bt}\right)$ with $r=1$, $a=1$, and $b=2+\zeta$.}\label{tab:galringout}
\end{table}
\end{landscape}

\end{example}

\subsection*{Acknowledgments} We would like to thank Marcin Mazur for his many helpful suggestions.

\bibliographystyle{alpha}

\end{document}